\documentclass[11pt]{article}

\usepackage{url,amsmath,amssymb,amsfonts,tikz}
\usepackage{enumerate}
\usepackage{color}
\usepackage{hyperref}
\newtheorem{theorem}{\bf Theorem}[section]
\newtheorem{corollary}[theorem]{\bf Corollary}

\newtheorem{lemma}[theorem]{\bf Lemma}

\newtheorem{proposition}[theorem]{\bf Proposition}

\newcommand{\proof}{\noindent{\bf Proof.\ }}
\newcommand{\qed}{\hfill $\Box$ \bigskip}
\newcommand{\cp}{\,\square\,}
\newcommand{\diam}{{\rm diam}}

\begin{document}

\title{On $\ell$-distance balanced product graphs}

\author{Janja Jerebic $^{a,b}$
\and
Sandi Klav\v zar $^{b,c,d}$
\and
Gregor Rus $^{a,d}$
}

\date{}

\maketitle

\begin{center}
$^a$ Faculty of Organizational Sciences, University of Maribor, Slovenia\\
\medskip

$^b$ Faculty of Natural Sciences and Mathematics, University of Maribor, Slovenia \\
\medskip

$^c$ Faculty of Mathematics and Physics, University of Ljubljana, Slovenia \\
\medskip

$^{d}$ Institute of Mathematics, Physics and Mechanics, Ljubljana, Slovenia \\
\medskip
\end{center}

\begin{abstract}
A graph $G$ is $\ell$-distance-balanced if for each pair of vertices $x$ and $y$ at distance $\ell$ 
in $G$, the number of vertices closer to $x$ than to $y$ is equal to the number of vertices closer 
to $y$ than to $x$. A complete characterization of $\ell$-distance-balanced corona products is given 
and a characterization of lexicographic products for $\ell \ge 3$, thus complementing known results 
for $\ell\in \{1,2\}$ and correcting an earlier related assertion. A sufficient condition on $H$ 
which guarantees that $K_n\cp H$ is $\ell$-distance-balanced is given and it is proved that if 
$K_n\cp H$ is $\ell$-distance-balanced, then $H$ is an $\ell$-distance-balanced graph. A known 
characterization of $1$-distance-balanced graphs is extended to $\ell$-distance-balanced graphs, 
again correcting an earlier claimed assertion.
\end{abstract}

\noindent
{\bf E-mails}: janja.jerebic@um.si, sandi.klavzar@fmf.uni-lj.si, gregor.rus4@um.si

Cartesian products; diameter-$2$ graphs; complete graphs

\medskip\noindent
{\bf AMS Subj. Class.}: 05C12, 05C76

\section{Introduction}
\label{sec:intro}

The investigation of distance-balanced graphs was initiated over twenty years ago in~\cite{handa-99}, an explicit definition of the concept was however given only a decade later in~\cite{jeklara-2008}. Distance-balanced graphs have since then been extensively studied by many authors from various points of view. On one side they were considered from the pure graph theoretical point of view~\cite{bcpss-2009, calu-2011, kmmm-2006, mispa-2012, polat-2019}. On the other hand they found significant applications in other areas, such as mathematical chemistry, communication networks, game theory, strategic interaction models, and elsewhere, see~\cite{bbckvzp-2014, ikm-2010, jeklara-2008, kyaaw-2009, khpo-2015}. We also refer to~\cite{cavaleri-2019} for a nice description of some of these applications as well as for connections between distance-balanced graphs and wreath products. Among many appealing results on distance-balanced graphs we point out that the class of distance-balanced graphs coincides with self-median graphs~\cite{bcpss-2009} and that they can also be characterized as the graphs whose opportunity index is zero~\cite{bbckvzp-2014}. Moreover, in mathematical chemistry the so-called Mostar index was introduced in~\cite{doslic-2018} as a  measure of how far a given graph is from being distance-balanced, see also~\cite{deng-2020, tepeh-2019}.

Considerable effort has been devoted to explore different generalizations of distance-balanced graphs, where one still focuses just on pairs of adjacent vertices~\cite{bcpss-2009, cado-2019, kmmm-2009, 
kuma-2014}. In addition, there is a very natural generalization of distance-balancedness to pairs of nonadjacent vertices. This idea can be traced back to the thesis of Frelih~\cite{fre-2014}, where $\ell$-distance-balanced graphs are introduced such that $1$-distance-balanced graphs coincide with distance-balanced graphs. 

Properties and general results on $\ell$-distance-balanced graphs have been discussed in several recently published papers. In particular, connected $2$-distance-balanced graphs which are not 2-connected, and $2$-distance-balanced graphs that can be represented as the Cartesian or the lexicographic product of two graphs were characterized in~\cite{fremi-2018}. In \cite{mispa-2018} infinitely many examples of $\ell$-distance-balanced graphs were presented, and $\ell$-distance-balanced graphs of diameter at most $3$ investigated in detail. Moreover, $\ell$-distance-balancedness of generalized Petersen graphs was analyzed. Now, the following~\cite[Problem 6.4]{mispa-2018} intrigued our attention: study $\ell$-distance-balanced graphs with respect to various graph products. In this paper we focus on the lexicographic, corona and Cartesian product which were already in the center of earlier investigations of 
$\ell$-distance-balanced graphs with respect to graph products.

Distance-balanced lexicographic product graphs were characterized in~\cite[Theorem 4.2]{jeklara-2008}, while one of the main objectives of~\cite{fapokh-2016} was to characterize $\ell$-distance-balanced lexicographic products for every positive integer 
$\ell$. But~\cite[Theorem 3.4]{fapokh-2016} is not correct for $\ell\ge 2$. For $\ell = 2$, the result was corrected 
in~\cite[Theorem 5.4]{fremi-2018}. Here, in  Section~\ref{sec:lexicographic}, we do the same for every $\ell \ge 3$. 
Corona product graphs in association with distance-balanced property have been (according to our knowledge) studied 
only in~\cite{tayo-2011}. It is known that the corona product of nontrivial, connected graphs is never distance-balanced. In Section~\ref{sec:corona} we characterize $\ell$-distance-balanced corona product graphs for every $\ell \geq 2$. Next, $1$-distance-balanced and $2$-distance-balanced Cartesian product graphs were characterized in~\cite[Proposition 4.1]{jeklara-2008} and~\cite[Theorem 4.4]{fremi-2018}, respectively. The difficulty of going from the first to the second result indicates that it might be very difficult to characterize $\ell$-distance-balanced Cartesian products for arbitrary $\ell$. In Section~\ref{sec:cartesian} we hence restrict ourselves to the case when one factor is complete. We give a sufficient condition on $H$ which guarantees that $K_n\cp H$ is $\ell$-distance-balanced and prove that if $K_n\cp H$ is $\ell$-distance-balanced, then $H$ is $\ell$-distance-balanced graph. In Section~\ref{sec:characterize} we give a characterization of $\ell$-distance-balanced graphs which extends the case $\ell=1$ from~\cite[Proposition 2.1]{jeklara-2008} and corrects the general case from~\cite[Proposition 2.2]{fapokh-2016}. 
Before giving our results, basic concepts used in this paper are introduced in the next section.

\section{Preliminaries}
\label{sec:prem}

In this section we introduce our notation and basic definitions. Throughout this paper, all graphs are simple, connected, undirected and finite. For a graph $G$, let $V(G)$ denote the set of vertices and $E(G)$ the set of edges of $G$. If $g_1,g_2\in V(G)$, then set
\begin{align*}
W_{g_1g_2} & = \{ g\in V(G):\ d_G(g,g_1) < d_G(g,g_2)\}\,, \\
{}_{g_1}\hspace{-1mm}W_{g_2} & = \{ g\in V(G):\ d_G(g,g_1)=d_G(g,g_2)\}\,,
\end{align*}
where $d_G(g_1,g_2)$ or simply $d(g_1,g_2)$ denotes the geodesic distance in $G$. 
In other words, $W_{g_1g_2}$ is the set of vertices in $G$ that are closer to 
$g_1$ than to $g_2$. The {\em diameter} $\diam(G)$ of a connected graph $G$ is the maximum distance between pairs of vertices of $G$.  If $\ell$ is a positive integer and $\diam(G)\geq \ell$, then we say that $G$ is {\em $\ell$-distance-balanced} if for any pair of vertices $g_1,g_2\in V(G)$ 
with $d_G(g_1,g_2)=\ell$ we have $|W_{g_1g_2}|=|W_{g_2g_1}|$. If the last equality 
holds for every $1\leq\ell\leq\diam(G)$, we say that $G$ is {\em highly distance-balanced}.
For instance, cycles and complete graphs are simple examples of such graphs. In addition,
every distance-regular graph is highly distance-balanced \cite{brocone-1998}. For more 
results on highly distance-balanced graphs see \cite{mispa-2018}.

Let $G\cp H$ and $G[H]$ respectively denote the {\em Cartesian product} and the {\em lexicographic product} of graphs $G$ and $H$. Both these  graph products have the vertex set $V(G) \times V(H)$. 
Vertices $(g_1,h_1)$ and $(g_2,h_2)$ are adjacent in $G\cp H$ if either $g_1=g_2$ and 
$h_1h_2\in E(H)$, or $h_1=h_2$ and $g_1g_2\in E(G)$. If $h\in V(H)$, then the subgraph 
of $G\cp H$ induced by the vertices $(g,h)$, $g\in V(G)$, is a {\em $G$-layer} and is 
denoted by $G^h$. Analogously $H$-layers $^gH$ are defined. $G$-layers and $H$-layers 
are isomorphic to $G$ and to $H$, respectively. Recall that
\begin{eqnarray}
\label{distance_cartesian}
d_{G\cp H}((g_1,h_1),(g_2,h_2))=d_G(g_1,g_2)+d_H(h_1,h_2).
\end{eqnarray}

Vertices $(g_1,h_1)$ and $(g_2,h_2)$ are adjacent in $G[H]$ if $g_1g_2\in E(G)$ or if 
$g_1=g_2$ and $h_1h_2\in E(H)$. The distance between two different vertices $(g_1,h_1)$ 
and $(g_2,h_2)$ in $G[H]$ for $G\neq K_1$ is determined as follows:
\begin{eqnarray}
\label{distance_lexicographic}
d_{G[H]}((g_1,h_1),(g_2,h_2))=\left\{\begin{array}{ll}
d_G(g_1,g_2); & g_1\neq g_2\,, \\
1;            & g_1=g_2 \quad{\rm and} \quad h_1h_2\in E(H)\,,\\
2;            & g_1=g_2 \quad{\rm and} \quad h_1h_2\notin E(H)\,.
\end{array}\right.
\end{eqnarray}

The {\em corona product} $G\circ H$ of graphs $G$ and $H$  is a graph obtained by taking one copy of $G$ and $|V(G)|$ copies of $H$ and joining each vertex of the $i$-th copy of $H$ with the $i$-th vertex of $G$. The vertex set of $G\circ H$ can therefore be written as $V(G\circ H) = \{(g,h): g \in V(G), h \in V(H) \cup \{0\}\}$, where the vertices $(g,0)$, $g\in V(G)$, correspond to the vertices of a copy of $G$ in $G\circ H$.

\section{On $\ell$-distance-balanced lexicographic products}
\label{sec:lexicographic}

As already explained in the introduction, $1$-distance-balanced lexicographic product graphs and $2$-distance-balanced lexicographic product graphs were characterized in~\cite[Theorem 4.2]{jeklara-2008} and in~\cite[Theorem 5.4]{fremi-2018}, respectively. In this section we give a characterization of $\ell$-distance-balanced lexicographic products for $\ell \ge 3$. This corrects~\cite[Theorem 3.4]{fapokh-2016} where a redundant condition of local regularity is required for the second factor. We begin with the following lemma needed for the announced characterization. 

\begin{lemma}
\label{lemma_lex}
Let $x=(g_1,h_1)$ and $y=(g_2,h_2)$ be arbitrary vertices of $\Gamma=G[H]$ with 
$d_G(g_1,g_2)=\ell\geq 3$. Then
$$|W_{xy}|=|W_{g_1g_2}|\cdot |V(H)|\,.$$
\end{lemma}

\proof
It follows from the assumption $d_G(g_1,g_2) \geq 3$ and from~\eqref{distance_lexicographic} that for any $h\in V(H)$ we have $(g_1,h)\in W_{xy}$ and $(g_2,h)\in W_{yx}$. Furthermore, if $g\in V(G)\setminus\{g_1,g_2\}$, then $d_{\Gamma}(x,(g,h))=d_G(g_1,g)$ and 
$d_{\Gamma}(y,(g,h))=d_G(g_2,g)$. Hence, $(g,h)\in W_{xy}$ if and only if $g\in W_{g_1g_2}$.
\qed

The announced characterization now reads as follows. 

\begin{theorem}
\label{theorem_lex}
Let $\ell \geq 3$ and $G\neq K_1$. Then $G[H]$ is $\ell$-distance-balanced if and only if
$G$ is $\ell$-distance-balanced.
\end{theorem}

\proof
Suppose $\Gamma=G[H]$ is $\ell$-distance-balanced and let $g_1,g_2\in V(G)$ be vertices with 
$d_G(g_1,g_2)=\ell$. For arbitrary chosen vertices $h_1,h_2\in V(H)$  
we denote $x=(g_1,h_1)$ and $y=(g_2,h_2)$. Then we have
$$d_{\Gamma}(x,y)=d_{\Gamma}((g_1,h_1),(g_2,h_2))=d_G(g_1,g_2)=\ell\,,$$ 
and consequently $|W_{xy}|=|W_{yx}|$. Since the vertices $x$ and $y$
meet the conditions of Lemma~\ref{lemma_lex}, we get 
$$|W_{g_1g_2}|\cdot |V(H)|=|W_{g_2g_1}|\cdot |V(H)|$$ which implies
$|W_{g_1g_2}|=|W_{g_2g_1}|$ and therefore confirms that $G$ is $\ell$-distance-balanced.

Conversely, assume $G$ is $\ell$-distance-balanced and examine any pair
of vertices $x=(g_1,h_1)$ and $y=(g_2,h_2)$ in $\Gamma$ with $d_{\Gamma}(x,y)=\ell\geq 3$.
Then we have
$$\ell=d_{\Gamma}(x,y)=d_{\Gamma}((g_1,h_1),(g_2,h_2))=d_G(g_1,g_2),$$
where the last equality holds by distance formula~\eqref{distance_lexicographic}. 
Lemma~\ref{lemma_lex} then implies
$$|W_{xy}|=|W_{g_1g_2}|\cdot |V(H)|=|W_{g_2g_1}|\cdot |V(H)|=|W_{yx}|\,.$$
Thus, $\Gamma=G[H]$ is $\ell$-distance-balanced.
\qed

\section{On $\ell$-distance-balanced corona products}
\label{sec:corona}

The corona product of two arbitrary, nontrivial and connected graphs is not 
distance-balanced~\cite[Theorem 3]{tayo-2011}. This implies that the corona
product of graphs $G$ and $H$ is distance-balanced if and only if $G$ is 
trivial ($G\cong K_1$) and $H$ is a complete graph (complete graphs are 
distance-balanced). In this section we give a characterization of 
$\ell$-distance-balanced corona products for $\ell \ge 2$. Note that if $G$ is a connected graph on at least two vertices, then $\diam(G\circ H) = \diam (G) + 2$. Hence we wish to know whether $G\circ H$ is $\ell$-distance-balanced for every $\ell \in \{2,\ldots, \diam(G)+2\}$.

We first consider $2$-distance-balanced corona products, for which the following concept is useful. 
A graph $G$ is \textit{locally regular} if any non-adjacent vertices
of $G$ have the same degree. Note that every regular graph is locally regular and that the converse does not hold. For example, complete bipartite graphs $K_{m,n}$, $m\neq n$, and wheel graphs $W_n$, $n\geq 5$, are locally regular but not regular.

\begin{proposition}
\label{prop:2DB}
Let $G$ be a connected graph and let $H$ be a graph with $|V(H)|\geq 2$. Then $G\circ H$ is $2$-distance-balanced if and only if $G\cong K_1$ and $H$ is locally regular.
\end{proposition}

\begin{proof}
Let $G\cong K_1$ and let $H$ be a locally regular graph. If $x,y$ are vertices of
$G\circ H$ with $d_{G\circ H}(x,y)=2$, then $x = (g,h_1)$ and $y =(g,h_2)$, 
where $d_H(h_1,h_2)\geq 2$. Hence, 
$W_{xy} = \{x\} \cup \{(g,h): h\in V(H), h_1h\in E(H), h_2 h \not \in E(H)\}$ and 
$W_{yx} = \{y\} \cup \{(g,h): h\in V(H), h_2h\in E(H), h_1 h \not \in E(H)\}$. The equality 
${\rm deg}(h_1) = {\rm deg}(h_2)$ then implies $|W_{xy}|=|W_{yx}|$.

Suppose now that $G\circ H$ is 2-distance-balanced and consider the vertices 
$x = (g_1,0)$ and $y = (g_2,h_2)$ for $g_1g_2 \in E(G)$ and $h_2 \in V(H)$.
Note that $d_{G\circ H}(x,y)=2$. Then $\{(g_1,h): h\in V(H) \cup \{0\}\}\subseteq W_{xy}$ 
and hence $|W_{xy}|\geq |V(H)|+1$. On the other hand we have 
$|W_{yx}| = |\{(g_2,h): h\in V(H), d_H(h,h_2)\le 1\}|\leq |V(H)|$. As this is not possible, we conclude that  $G\cong K_1$.

In the sequel, let $G = K_1$ and $V(G) = \{g\}$.
 Consider now the vertices  $x = (g,h_1)$ and $y = (g,h_2)$ of $G\circ H$ for $h_1,h_2 \in V(H)$ 
with $d_H(h_1,h_2) \ge 2$. Note that $d_{G\circ H}(x,y)=2$. 
Then $W_{xy} = \{x\} \cup \{(g,h):  h\in V(H), hh_1 \in E(H), hh_2 \not \in E(H)\}$ 
and similarly $W_{yx} = \{y\} \cup \{(g,h): h\in V(H), hh_2 \in E(H), hh_1 \not \in E(H)\}$. 
Since $|W_{xy}|=|W_{yx}|$, we conclude that $H$ is locally regular.
\qed
\end{proof}

Proposition~\ref{prop:2DB} immediately gives the following characterization of $2$-distance balanced graphs that contain a universal vertex, where a vertex $u$ of a graph $G$ is {\em universal} if its degree is $|V(G)|-1$. 

\begin{corollary}
Let $v$ be a universal vertex of a graph $G$. Then $G$ is $2$-distance-balanced if and only if $G-v$ is locally regular.  
\end{corollary}

Because of Proposition~\ref{prop:2DB} and since $\diam(K_1\circ H) \in \{1,2\}$, we are next interested only in corona products $G\circ H$, where $G$ is a connected graph of order at least $2$. 

\begin{lemma}
\label{thm:ndb}
Let $G$ be a connected graph with at least two vertices, $H$ a graph, and $3\le \ell \le \diam(G)+2$.  
Then $G\circ H$ is $\ell$-distance-balanced if and only if 
the following conditions are fulfilled. 
\begin{itemize}
\item[(i)] $G$ is $\ell$-distance-balanced, 
\item[(ii)] $G$ is $(\ell-2)$-distance-balanced, and
\item[(iii)] $|\{g \in V(G): d_G(g_1,g)+2\le d_G(g_2,g)\}| = |\{g \in V(G): d_G(g_2,g)\le d_G(g_1,g)\}|$
for every $g_1,g_2 \in V(G)$ with $d_G(g_1,g_2) = \ell-1$.
\end{itemize}
\end{lemma}

\begin{proof}
Suppose that $G\circ H$ is $\ell$-distance-balanced. 
Consider vertices  $x = (g_1,h_1)$ and $y = (g_2,h_2)$ of $G\circ H$ with $d_{G\circ H}(x,y)=\ell$.
Then there are three cases to be considered. 

\medskip\noindent
{\bf Case 1.} $h_1=h_2=0$.\\
In this case we have $d_G(g_1,g_2) = \ell$.
For $z = (g_3,h_3) \in W_{xy}$ we have $d_G(g_1,g_3)<d_G(g_2,g_3)$ and similarly 
$z \in W_{yx}$ implies $d_G(g_1,g_3)>d_G(g_2,g_3)$. Since $|W_{xy}| = |W_{yx}|$, 
this means that $|W_{g_1g_2}| = |W_{g_2g_1}|$ and therefore $G$ is $\ell$-distance-balanced.

\medskip\noindent
{\bf Case 2.} $h_1\neq 0$ and $h_2 \neq 0$.\\
Now we have $d_G(g_1,g_2) = \ell-2$. 
If $z = (g_3,h_3) \in W_{xy}$, then $d_G(g_1,g_3)<d_G(g_2,g_3)$. Similarly, if
$z \in W_{yx}$, then $d_G(g_1,g_3)>d_G(g_2,g_3)$. Since $|W_{xy}| = |W_{yx}|$, this means that 
$|W_{g_1g_2}| = |W_{g_2g_1}|$ and therefore $G$ is $(\ell-2)$-distance-balanced.

\medskip\noindent
{\bf Case 3.} $h_1 \neq 0$ and $h_2 = 0$. \\
In this case, $d_G(g_1,g_2) = \ell-1$.
Again let $z = (g_3,h_3)$ be a vertex of $G\circ H$. If $z \in W_{xy}$, then $d_G(g_1,g_3)+1 <d_G(g_2,g_3)$. 
On the other hand, $z \in W_{yx}$ implies that $d_G(g_1,g_3)\geq d_G(g_2,g_3)$. Since $|W_{xy}| = |W_{yx}|$, 
it follows that $|\{g \in V(G): d_G(g_1,g)+2\le d_G(g_2,g)\}| = |\{g \in V(G): d_G(g_2,g)\le d_G(g_1,g)\}|$.

We have thus proved that if $G\circ H$ is $\ell$-distance-balanced, then (i), (ii), and (iii) hold. The reverse implication is clear.
\qed
\end{proof}

\begin{theorem}
\label{thm:ndb.char}
If $G$ is a connected graph with at least two vertices, and $H$ is a graph, then the following hold. 

(i) $G\circ H$ is $(\diam(G) +2)$-distance-balanced if and only if $G$ is $\diam(G)$-distance-balanced.

(ii) If $\ell \in \{3,\ldots, \diam(G)+1\}$, then $G\circ H$ is not $\ell$-distance-balanced. 
\end{theorem}

\begin{proof}
(i) If $x = (g_1,h_1)$ and $y = (g_2,h_2)$ are vertices of $G\circ H$ with $d_{G\circ H}(x,y) = (\diam(G) +2)$, then $h_1\neq 0$ and $h_2 \neq 0$. Hence we only need to consider Case~2 of Lemma~\ref{thm:ndb} which implies the assertion (i).

(ii) Let $\ell \in \{3,\ldots, \diam(G)+1\}$. To prove that $G\circ H$ is not $\ell$-distance-balanced, in view of Lemma~\ref{thm:ndb} it suffices to prove the following: 

\medskip\noindent
{\bf Claim}: If  $X$ is a connected graph and $u,v \in V(X)$ with $d_X(u,v)=k\ge 2$, then 
$$|\{x \in V(X): d_X(u,x) +2 \leq d_X(v,x)\}| \ne |\{x \in V(X): d_X(v,x) \leq d_X(u,x)\}\,.$$
Consider the following sets:
\begin{align*}
U_2 & = \{x \in V(X): d_G(u,x)  \le d_G(v,x) - 2\}\,, \\
U_1 & = \{x \in V(G): d_G(u,x)  = d_G(v,x) - 1\}\,, \\
E & = \{x \in V(G): d_G(u,x) = d_G(v,x)\}\,, \\
V_1 & = \{x \in V(G): d_G(u,x) = d_G(v,x) + 1\}\,, \\
V_2 & = \{x \in V(G): d_G(u,x) \ge d_G(x,x) + 2\}\,. 
\end{align*}

Clearly, every vertex of $X$ is contained in exactly one of the above sets. By way of contradiction suppose that the equality holds in the displayed formula of the claim. Then 
$|U_2| = |E| + |V_1| + |V_2|$ and $|V_2| = |E| + |U_1| + |U_2|$.
It follows that $|E| = |U_1| = |V_1| = 0$. Consider a shortest $u,v$-path $P$. 
If $k$ is even, then $P$ contains a vertex $x$  such that $d_X(u,x) = d_X(v,x)$.  
This means that $x \in E$, and so $|E| \neq 0$. Consequently $k$ must be odd.
But if $k$ is odd, then there exist vertices $x$ and $y$ on $P$ such that 
$d_X(u,x)=d_X(v,x)-1$ and similarly $d_X(u,y)=d_X(v,y)+1$, which implies that $x \in U_1$ 
and $y \in V_1$.  This contradiction proves the claim which in turn yields (ii).
\qed
\end{proof}

\section{On $\ell$-distance-balanced Cartesian products}
\label{sec:cartesian}

As already explained, $1$-distance-balanced and $2$-distance-balanced Cartesian product graphs were characterized in~\cite{jeklara-2008} and~\cite{fremi-2018}, respectively. As the general case seems difficult, we reduce here our attention to the case where one factor is complete. In the following lemma we first analyze and present the conditions for $x,y\in V(K_n\cp H)$ under which the vertices of $K_n\cp H$ are contained in $W_{xy}$. 

\begin{lemma}
\label{lemma_cart}
Let $x=(g_1,h_1)$ and $y=(g_2,h_2)$ be arbitrary vertices of $\Gamma=K_n\cp H$, $n\geq 2$. 
Then the following holds:
\begin{enumerate}
\item[(i)] If $x$ and $y$ are contained in the same $H$-layer ($g_1=g_2$), 
      then the set $W_{xy}$ contains exactly the vertices $z=(g,h)\in \Gamma$ 
			for which $h\in W_{h_1h_2}$.
\item[(ii)] If $x$ and $y$ are not contained in the same $H$-layer ($g_1\neq g_2$) and $z=(g,h)$ is a vertex of $\Gamma$ contained in
      \begin{itemize}
			\item ${}^{g_1}\hspace{-1mm}H$, then $z\in W_{xy}\iff h\in(W_{h_1h_2}\cup {}_{h_1}\hspace{-1mm}W_{h_2})$.
			\item ${}^{g_2}\hspace{-1mm}H$, then $z\in W_{xy}\iff h\in W_{h_1h_2}$ and $d_H(h_1,h)\neq d_H(h_2,h)-1$.
			\item $({}^{g_1}\hspace{-1mm}H\cup {}^{g_2}\hspace{-1mm}H)^c$, then $z\in W_{xy}\iff h\in W_{h_1h_2}$.
			\end{itemize}	
\end{enumerate}
\end{lemma}

\proof
Note that for a complete graph $G$ the distance formula~\eqref{distance_cartesian} 
can be simplified as $d_{G\cp H}((g_1,h_1),(g_2,h_2))=\delta_{g_1,g_2}+d_H(h_1,h_2)$,
where $\delta_{g_1,g_2}$ is 0 or 1 depending on whether $g_1=g_2$ or not, respectively.

For a vertex $z=(g,h)$ of $\Gamma$ we have
\begin{eqnarray*}
z\in W_{xy}  & \iff & \delta_{g_1,g}+d_H(h_1,h)<\delta_{g_2,g}+d_H(h_2,h)\,.
\end{eqnarray*} 
If $x$ and $y$ are contained in the same $H$-layer, we obtain $z\in W_{xy}$ if and only if $h\in W_{h_1h_2}$. 
Thus, $(i)$ follows.

Suppose now that $x$ and $y$ are not contained in the same $H$-layer. 
For $z\in({}^{g_1}\hspace{-1mm}H\cup {}^{g_2}\hspace{-1mm}H)^c$ we have $\delta_{g_1,g}=\delta_{g_2,g}=1$
and therefore $z\in W_{xy}$ if and only if $h\in W_{h_1h_2}$. 
For $z\in {}^{g_1}\hspace{-1mm}H$ we have $\delta_{g_1,g}=0$ and $\delta_{g_2,g}=1$ and hence 
$z\in W_{xy}$ if and only if $ h\in (W_{h_1h_2}\cup {}_{h_1}\hspace{-1mm}W_{h_2})$.
Finally, let $z\in {}^{g_2}\hspace{-1mm}H$. Then $\delta_{g_1,g}=1$ and $\delta_{g_2,g}=0$ which
implies $z\in W_{xy}$ if and only if $h\in W_{h_1h_2}$ and $d_H(h_1,h)\neq d_H(h_2,h)-1$.
This completes the proof of $(ii)$.
\qed

\begin{theorem}
\label{theorem_cart1}
Let $n\geq 2$, $\ell \geq 2$, and let $H$ be $\ell$-distance-balanced and $(\ell-1)$-distance-balanced 
graph. Then $K_n\cp H$ is $\ell$-distance-balanced if and only if
\begin{eqnarray}
\label{pogoj_cp}
|\{h\in W_{h_1h_2}:\ d(h_1,h) = d(h_2,h)-1\}| = |\{h\in W_{h_2h_1}:\ d(h_2,h) = d(h_1,h)-1\}|
\end{eqnarray}
for every $h_1,h_2\in V(H)$ with $d_H(h_1,h_2)=\ell-1$.
\end{theorem}

\proof
Assume first that $H$ meets the condition~\eqref{pogoj_cp} of the theorem and let $x=(g_1,h_1)$ and 
$y=(g_2,h_2)$ be arbitrary vertices of $\Gamma=K_n\cp H$ with $d_{\Gamma}(x,y)=\ell$. 
Note that for $g_1=g_2$ we have $\ell=d_{\Gamma}(x,y)=d_H(h_1,h_2)$. Moreover, 
Lemma~\ref{lemma_cart} implies that $|W_{xy}|=n\cdot|W_{h_1h_2}|$ and 
$|W_{yx}|=n\cdot|W_{h_2h_1}|$. Considering that $H$ is $\ell$-distance-balanced, 
we can conclude, that $|W_{xy}|=|W_{yx}|$. Suppose now that $g_1\neq g_2$. Then
$\ell=d_{\Gamma}(x,y)=1+d_H(h_1,h_2)$ and hence $d_H(h_1,h_2)=\ell-1$. Since $H$ is
$(\ell-1)$-distance-balanced and satisfies the condition~\eqref{pogoj_cp}, 
Lemma~\ref{lemma_cart} implies that 
\begin{eqnarray*}
|W_{xy}| & = & n\cdot|W_{h_1h_2}|+|{}_{h_1}\hspace{-1mm}W_{h_2}| - |\{h\in W_{h_1h_2}:\ d_H(h_1,h) = d_H(h_2,h)-1\}|\\ 
         & = & n\cdot|W_{h_2h_1}|+|{}_{h_1}\hspace{-1mm}W_{h_2}| - |\{h\in W_{h_2h_1}:\ d_H(h_2,h) = d_H(h_1,h)-1\}|\\ 
				 & = & |W_{yx}|\,.
\end{eqnarray*}
Therefore, $\Gamma$ is $\ell$-distance-balanced.

For the converse let $h_1,h_2$ be any vertices of $V(H)$ with $d_H(h_1,h_2)=\ell-1$. 
Consider now the vertices $x=(g_1,h_1)$ and $y=(g_2,h_2)$ of $K_n\cp H$ with $g_1\neq g_2$.
Then by Lemma~\ref{lemma_cart} the equality~\eqref{pogoj_cp} holds.
\qed

We next show a necessary condition for $K_n\cp H$ to be $\ell$-distance-balanced. 

\begin{proposition}
\label{proposition_cart2}
Let $H$ be a graph of diameter at least $\ell\geq 2$ and let $n\geq 1$. If the 
Cartesian product $K_n\cp H$ is $\ell$-distance-balanced, then $H$ is 
$\ell$-distance-balanced.
\end{proposition}

\proof
Suppose $\Gamma=K_n\cp H$ is $\ell$-distance-balanced. Let $h_1$ and $h_2$ be 
arbitrary vertices of $H$ with $d_H(h_1,h_2)=\ell$ and let $g$ be any vertex of $K_n$. 
Then $\ell=d_H(h_1,h_2)=d_{\Gamma}((g,h_1),(g,h_2))$. Since $\Gamma$ is $\ell$-distance-balanced 
we have $|W_{(g,h_1)(g,h_2)}|=|W_{(g,h_2)(g,h_1)}|$. Using Lemma~\ref{lemma_cart} we
derive that $n\cdot|W_{h_1h_2}|=n\cdot|W_{h_2h_1}|$ whence it follows that 
$|W_{h_1h_2}|=|W_{h_2h_1}|$. Therefore, $H$ is $\ell$-distance-balanced graph. 
\qed

From Lemma~\ref{lemma_cart}, Theorem~\ref{theorem_cart1}, and Proposition~\ref{proposition_cart2} we can deduce: 

\begin{corollary}
\label{cor_cart_2}
Let $H$ be a graph and let $n \geq 2$. Then $K_n\cp H$ is $2$-distance-balanced 
if and only if $H$ is a $2$-distance-balanced and $1$-distance-balanced graph.
\end{corollary}

\proof
Assume first that $H$ is $2$-distance-balanced and $1$-distance-balanced graph.
Let $h_1$ and $h_2$ be any adjacent vertices of $H$. Then the condition~\eqref{pogoj_cp}
of Theorem~\ref{theorem_cart1} coincides with $1$-distance-balancedness of $H$ which
implies that $K_n\cp H$ is $2$-distance-balanced.

Suppose now that $\Gamma=K_n\cp H$ is $2$-distance-balanced graph. According to 
Proposition~\ref{proposition_cart2} then also $H$ is $2$-distance-balanced. 
It remains to show that in addition $H$ is $1$-distance-balanced. Let 
$h_1,h_2\in V(H)$ be adjacent vertices, and let $g_1$ and $g_2$ be different vertices 
of $K_n$. Consider now the vertices $x=(g_1,h_1)$ and $y=(g_2,h_2)$ of $\Gamma$.
By Lemma~\ref{lemma_cart} we obtain
$$|W_{xy}|=(n-1)|W_{h_1h_2}|+|{}_{h_1}\hspace{-1mm}W_{h_2}|$$ and
$$|W_{yx}|=(n-1)|W_{h_2h_1}|+|{}_{h_1}\hspace{-1mm}W_{h_2}|\,.$$
Since $d_{\Gamma}(x,y)=2$ and $\Gamma$ is $2$-distance-balanced we have $|W_{xy}|=|W_{yx}|$ which 
completes the proof.
\qed

Corollary~\ref{cor_cart_2} can alternatively be deduced also from~\cite[Theorem 4.4]{fremi-2018}. 

\section{A characterization of $\ell$-distance-balanced graphs}
\label{sec:characterize}

If $G$ is a graph and $k$ a non-negative integer, then let 
$N_k(x) = \{y:\ d(x,y) = k\}$ and $N_k[x] = \{y:\ d(x,y) \leq k\}$. (Recall that $|N_1(x)|$ is the {\em degree} $\deg(x)$ of the vertex $x$.) In~\cite[Proposition 2.1]{jeklara-2008} it was proved that a graph $G$ of diameter $d$ is distance-balanced if and only if  
$$|N_1[a]\setminus N_1[b]|+ \sum_{k=2}^{d-1}|N_k(a)\setminus N_{k-1}(b)|=
|N_1[b]\setminus N_1[a]|+\sum_{k=2}^{d-1}|N_k(b)\setminus N_{k-1}(a)|$$
holds for every edge $ab\in E(G)$. An attempt to generalize this result  to $\ell$-distance-balanced graphs was given in~\cite[ Proposition 2.2]{fapokh-2016}. However, counterexamples were presented in \cite[Remark 4.3]{mispa-2018}. We now give an accordingly modified version of the result. 

\begin{proposition}
\label{char}
A graph $G$ of diameter $d$ is $\ell$-distance-balanced $(1\leq \ell \leq d)$ if and only if 
$$\sum_{k=1}^{d-1}|N_k(a)\setminus N_{k-1}[b]|=\sum_{k=1}^{d-1}|N_k(b)\setminus N_{k-1}[a]|$$
holds for all $a,b\in V(G)$ with $d(a,b)=\ell$.
\end{proposition}

\proof
Let $a$ and $b$ be arbitrary vertices of $G$ with $d(a,b)=\ell$. 
Then $W_{ab}$ and $W_{ba}$ can be written as 
$$W_{ab} = \{a\} \cup \bigcup_{k=1}^{d-1} (N_k(a)\setminus N_{k}[b])=
           \{a\} \cup \bigcup_{k=1}^{d-1} \bigg((N_k(a)\setminus N_{k-1}[b]) \setminus (N_k(a)\cap N_k(b))\bigg)$$
and 
$$W_{ba}=  \{b\} \cup \bigcup_{k=1}^{d-1} (N_k(b)\setminus N_{k}[a])=
           \{b\} \cup \bigcup_{k=1}^{d-1} \bigg((N_k(b)\setminus N_{k-1}[a]) \setminus (N_k(b)\cap N_k(a))\bigg)$$				
Since $N_k(a)\cap N_k(b)$ is a subset of both $N_k(a)$ and $N_k(b)$, the result follows.

\qed

\begin{corollary}
\label{2db-degree}
If $G$ is a graph of diameter $2$, then the following statements are equivalent. 
\begin{enumerate}[(i)]
\item $G$ is $2$-distance-balanced. 
\item $\deg_G(a) = \deg_G(b)$ for every $a,b\in V(G)$ with $d(a,b)=2$.
\item $G$ is a regular graph, or a nonregular join of at least two regular graphs. 
\end{enumerate}
\end{corollary}

\proof
The equivalence $(i) \Leftrightarrow (ii)$ easily follows from Proposition~\ref{char}, 
while the equivalence $(i) \Leftrightarrow (iii)$ was proved in~\cite[Theorem 4.2]{mispa-2018}. 
\qed

\section*{Acknowledgements}

We thank referees for numerous useful remarks which, in particular, enabled us to significantly shorten some of the arguments. We acknowledge the financial support from the Slovenian Research Agency 
(research core funding No.\ P1-0297 and projects J1-8130, J1-9109, J1-1693, N1-0095).

\end{document}